\documentclass[12pt]{article}

\usepackage{hyperref}

\usepackage{amsmath,amsthm,amsfonts,amscd,amssymb,bm,mathrsfs}
\usepackage{enumerate}
\usepackage{graphicx,xcolor,epstopdf,subfigure}
\usepackage{multirow,cases}
\usepackage{authblk}
\usepackage{geometry}
\usepackage[title]{appendix}
\DeclareMathOperator*{\argmin}{arg\,min}

\newcommand{\R}{{\mathbb{R}}}

\newcommand{\I}{{\mathbb{I}}}
\newcommand{\Su}{{\mathcal{S}}}
\newcommand{\loss}{{\mathcal{L}}}
\newcommand{\regu}{{\mathcal{P}}}
\newcommand{\obj}{\Psi}
\newcommand{\Cc}{{\mathcal{C}}}

\newcommand{\Aa}{{\mathbb{A}}}
\newcommand{\Bb}{{\mathbb{B}}}
\newcommand{\X}{{\mathscr{X}}}
\newcommand{\N}{{\mathcal{N}}}
\newcommand{\XX}{{\mathfrak{X}}}
\newcommand{\Z}{{\mathfrak{Z}}}

\newcommand{\Q}{{\mathcal{Q}}}
\newcommand{\normm}[1]{{\left\vert\kern-0.25ex\left\vert\kern-0.25ex\left\vert #1
\right\vert\kern-0.25ex\right\vert\kern-0.25ex\right\vert}}
\newcommand{\inm}[2]{\langle\langle #1,#2 \rangle\rangle}

\newtheorem{Assumption}{Assumption}
\newtheorem{Lemma}{Lemma}
\newtheorem{Definition}{Definition}
\newtheorem{Proposition}{Proposition}
\newtheorem{Remark}{Remark}
\newtheorem{Theorem}{Theorem}
\newtheorem{Corollary}{Corollary}

\title{Low-rank matrix recovery via nonconvex optimization methods with application to errors-in-variables matrix regression}
\author[1]{Xin Li}
\author[2]{Dongya Wu}
\affil[1]{School of Mathematics, Northwest University, Xi’an, 710069, P. R. China}
\affil[2]{School of Information Science and Technology, Northwest University, Xi’an, 710069, P. R. China}

\date{}

\begin{document}
\maketitle

\begin{abstract}
We consider the nonconvex regularized method for low-rank matrix recovery. Under the assumption on the singular values of the parameter matrix, we provide the recovery bound for any stationary point of the nonconvex method by virtue of regularity conditions on the nonconvex loss function and the regularizer. This recovery bound can be much tighter than that of the convex nuclear norm regularized method when some of the singular values are larger than a threshold defined by the nonconvex regularizer. In addition, we consider the errors-in-variables matrix regression as an application of the nonconvex optimization method. Probabilistic consequences and the advantage of the nonoconvex method are demonstrated through verifying the regularity conditions for specific models with additive noise and missing data.
\end{abstract}


\section{Introduction}\label{sec-intro}

We consider the low-rank matrix recovery problem, which appears in many applications such as image processing and collaborative filtering. Specifically, one aims to recover an unknown low-rank matrix $\Theta^*\in \R^{d_1\times d_2}$ with $\text{rank}(\Theta^*)\leq r\ll \min\{d_1,d_2\}$ from some measurements. A natural idea is to solve the following rank-constrained minimization problem
\begin{equation}\label{eq-rankmin}
\min_{\Theta\in \R^{d_1\times d_2}} \loss_N(\Theta)\quad \mbox{s.t.} \quad \text{rank}(\Theta)\leq r,
\end{equation}
where $\loss_N:\R^{d_1\times d_2}\rightarrow \R$ is usually a smooth and convex loss function measuring data fitting. Unfortunately, due to the nonconvex and combinational natures of \eqref{eq-rankmin}, it is NP-hard to obtain a global solution \cite{natarajan1995sparse}. Researchers have then devoted to seek for convex relaxations of \eqref{eq-rankmin}. A popular way is to use the convex nuclear norm $\normm{\cdot}_*$, which is the sum of all singular values of a matrix, as a surrogate for the rank function, and two optimization problems have been proposed \cite{Fazel2001ARM}. The first one is a constrained minimization problem
\begin{equation}\label{eq-nuclearmin}
\min_{\Theta\in \R^{d_1\times d_2}} \loss_N(\Theta)\quad \mbox{s.t.} \quad \normm{\Theta}_*\leq \zeta,
\end{equation}
where $\zeta$ is a tuning parameter,
and the second one is a regularization problem 
\begin{equation}\label{eq-nuclearregu}
\min_{\Theta\in \R^{d_1\times d_2}} \loss_N(\Theta)+\lambda\normm{\Theta}_*,
\end{equation}
where $\lambda$ is a regularization parameter. Both of problems \eqref{eq-nuclearmin} and \eqref{eq-nuclearregu} enjoy the benefits of convex optimization and can be solved in polynomial time by a number of numerical algorithms \cite{cai2010singular,Becker2011Templates,Lee2010ADMiRA}. In the theoretical aspect, global recovery bounds, which measure the distance between the global solutions and the true low-rank parameter $\Theta^*$, have also been established for \eqref{eq-nuclearmin} and \eqref{eq-nuclearregu} under suitable regularity conditions such as the restricted isometry property (RIP) and the restricted strong convexity (RSC)\cite{candes2011tight,negahban2011estimation,recht2010guaranteed}.

On the other hand, it is easy to see that the nuclear norm relaxation methods \eqref{eq-nuclearmin} and \eqref{eq-nuclearregu}, which penalize the singular values of a matrix, is a generalization of the $\ell_1$ norm penalized method in sparse linear regression. Hence the deficiency of the latter one has also been inherited that significant estimation bias is induced since larger and more informative singular values are penalized \cite{wang2014optimal,zhang2008sparsity}. In view of the advantages of nonconvex regularizers such as the smoothly clipped absolute deviation penalty (SCAD) \cite{fan2001variable} and the minimax concave penalty (MCP) \cite{Zhang2010} that they are able to achieve more refined recovery accuracy and variable selection consistency in linear regression \cite{wang2014optimal},  researchers have paid increasing attention to nonconvex regularizers imposed on the singular values of a matrix to perform low-rank approximation \cite{Mohan2012Iterative}. Some popular instances of nonconvex regularizers include the Schatten $\ell_p\ (0<p<1)$ norm \cite{rohde2011estimation}, the truncated nuclear norm \cite{Hu2013Fast}, the SCAD and MCP \cite{zhou2014regularized,Lu2014GeneralizedNN,gui2015towards}. Meanwhile, the loss function can also be nonconvex in real applications, such as error-in-variables matrix regression; see \cite{carroll2006measurement,Li2024LowrankME} and references therein.

Although empirical results have shown the superiority of nonconvex regularizers over the convex nuclear norm and extensive studies have been made on optimization methods to solve nonconvex regularized problems \cite{wang2021nonconvex,Hu2013Fast,yao2015fast,yao2017large}, little is unknown about the theoretical properties of nonconvex regularizes for low-rank matrix recovery except the work \cite{gui2015towards}.\cite{gui2015towards} presented a unified framework for low-rank matrix recovery with nonconvex regularizers, which satisfies some curvature and dominant property to control the nonconvexity level. The resulting nonconvex estimator is shown to enjoy a faster statistical convergence rate than that of the convex nuclear norm regularized estimator.

However, \cite{gui2015towards} only established the recovery bound for the global solution of the nonconvex regularized estimator, while the nonconvex optimization objective functions may have many local optima that are not global optima, a fact which might result in that algorithms such as gradient descent may attain undesired local optima. This weakness leads to a considerable gap between theory and practice. Moreover, empirical studies have also pointed out that local optima of a class of  nonconvex estimators arising in statistical inference problems possess good recovery performance \cite{breheny2011coordinate}. Therefore, it is necessary to analyse recovery bounds of local solutions in nonconvex low-rank matrix estimation problems.

The main purpose of this paper is to deal with a more general case where the loss function and the regularizer can both be nonconvex. The main contributions of this paper are as follows. Under the assumption on singular values of the true parameter matrix, we prove the recovery bound for any stationary point of the nonconvex optimization problem under suitable forms of the RSC condition. When some of the singular values are larger than a threshold defined by the nonconvex regularizer, this recovery bound is much tighter than that of the convex nuclear norm regularized method; see Theorem \ref{thm-stat}. This theoretical result is of practical importance since it is does not rely on any specific numerical algorithms, suggesting that any algorithm can consistently recover the true low-rank matrix as long as it converges to a stationary point.

Furthermore, in the aspect of application, high-dimensional error-in-variables matrix regression is considered as an instance of the nonconvex method. Measurement errors cannot be avoided in practice due to instrumental or economical constraints, and thus the collected data, such as genetic data, may always be noisy or partially missing. Worse still, methods for clean data cannot be naively applied otherwise, only misleading inference results can be obtained \cite{sorensen2015measurement}. There have been some results on errors-in-variables linear regression \cite{loh2012high,datta2017cocolasso,rosenbaum2010sparse,li2020sparse}, while for low-rank matrix recovery, only the particular case — multi-response models are considered with the convex nuclear norm as the regularizer \cite{wu2020scalable,li2023Lowrank}. However, little attention paid on the more general and widely-used errors-in-variables matrix regression model, which is a generic and unified observation model including many different models such as the multi-response model, matrix completion, matrix compressed sensing and so on.
As we have mentioned before, the loss function is generally nonconvex in errors-in-variables regression, and thus the parameter estimation problem falls into the framework of this paper. Hence the proposed nonconvex regularized method is applied on the errors-in-variables matrix regression model via verifying regularity conditions; see Corollaries \ref{corol-add} and \ref{corol-mis}.

The remainder of this article is organized as follows. In Sect. \ref{sec-prob}, we propose a general nonconvex  estimator based on nonconvex spectral regularization. Some regularity conditions are imposed on the loss function for further analysis. In Sect. \ref{sec-main}, we establish our main results on statistical recovery bounds. In Sect. \ref{sec-conse}, probabilistic consequences on the regularity conditions for errors-in-variables matrix models are obtained. Conclusions and future work are discussed in Sect. \ref{sec-con}. 

We end this section by introducing useful notations. For a vector $\beta\in \R^d$ and an index set $J\subseteq \{1,2,\dots,d\}$, we use $\beta_J$ to denote the vector in which $(\beta_J)_i=\beta_i$ for $i\in J$ and zero elsewhere, $\text{mat}(\beta)\in \R^{d\times d}$ to denote the diagonal matrix with diagonal elements equal to $\beta_i\ (i=1,2,\cdots,d)$, $|J|$ to denote the cardinality of $J$, and $J^c=\{1,2,\dots,d\}\setminus J$ to denote the complement of $J$. For $d\geq 1$, let $\I_d$ stand for the $d\times d$ identity matrix. For a matrix $X\in \R^{d_1\times d_2}$, let $X_{ij}\ (i=1,\dots,d_1,j=1,2,\cdots,d_2)$ denote its $ij$-th entry, $X_{i\cdot}\ (i=1,\dots,d_1)$ denote its $i$-th row, $X_{\cdot j}\ (j=1,2,\cdots,d_2)$ denote its $j$-th column, and vec$(X)\in \R^{d_1d_2}$ to denote its vectorized form. When $X$ is a square matrix, i.e., $d_1=d_2$, we use diag$(X)$ stand for the diagonal matrix with its diagonal elements equal to $X_{11},X_{22},\cdots,X_{d_1d_1}$. We write $\lambda_{\text{min}}(X)$ and $\lambda_{\text{max}}(X)$ to denote the minimal and maximum eigenvalues of a matrix $X$, respectively. For a matrix $\Theta\in \R^{d_1\times d_2}$, define $d=\min\{d_1,d_2\}$, $\tilde{d}=\max\{d_1,d_2\}$, and denote its singular values in decreasing order by $\sigma_1(\Theta)\geq \sigma_2(\Theta)\geq \cdots \sigma_d(\Theta)\geq 0$. We use $\normm{\cdot}$ to denote different types of matrix norms based on singular values, including the nuclear norm $\normm{\Theta}_*=\sum_{j=1}^{d}\sigma_j(\Theta)$, the spectral or operator norm $\normm{\Theta}_{\text{op}}=\sigma_1(\Theta)$, and the Frobenius norm $\normm{\Theta}_\text{F}=\sqrt{\text{trace}(\Theta^\top\Theta)}=\sqrt{\sum_{j=1}^{d}\sigma_j^2(\Theta)}$. For a pair of matrices $\Theta$ and $\Gamma$ with equal dimensions, we let $\inm{\Theta}{\Gamma}=\text{trace}(\Theta^\top \Gamma)$ denote the trace inner product on matrix space. For a function $f:\R^d\to \R$, $\nabla f$ is used to denote the gradient when $f$ is differentiable, and $\partial f$ is used to denote the subdifferential that consists of all subgradients when $f$ is nondifferentiable but convex.

\section{Problem setup}\label{sec-prob}

In this section, we propose a general nonconvex estimator to estimate a low-rank matrix, and then impose some regularity conditions on the nonconvex loss function and the nonconvex regularizer.

\subsection{General nonconvex estimator}

In this article, we mainly consider the high-dimensional scenario where the number of unknown matrix entries $d_1\times d_2$ can be much larger than the number of observations $N$. Researchers have already pointed out that consistent estimation cannot be achieved under this high-dimensional setting unless the model space is imposed with additional structures, such as low-rankness in matrix estimation problems \cite{negahban2012unified}. Empirical facts have also shown that low-rank matrices always arise in real applications, such as multi-task learning \cite{wu2019joint} and collaborative filtering \cite{sagan2021lowrank}. In the following, we shall impose the low-rank constraint on the parameter space. 

For a matrix $\Theta\in \R^{d_1\times d_2}$, let $d=\min\{d_1,d_2\}$ and $\sigma(\Theta)$ stand for the vector formed by the singular values of $\Theta$ in decreasing order, i.e., $\sigma(\Theta)=(\sigma_1(\Theta),\sigma_2(\Theta),\cdots,\sigma_d(\Theta))\top$ and $\sigma_1(\Theta)\geq \sigma_2(\Theta)\geq \cdots \sigma_{d}(\Theta)$. The true parameter $\Theta^*\in \R^{d_1\times d_2}$ is assumed to be of low-rank with 
\begin{equation}\label{eq-rank}
\text{rank}(\Theta^*)=r\ll \min\{d_1,d_2\}.
\end{equation} 
This low-rank assumption implies that there are only $r$ nonzero elements in the singular values of the true parameter $\Theta^*$.

Consider the following regularized $M$-estimator
\begin{equation}\label{eq-esti}
\hat{\Theta} \in \argmin_{\Theta\in \Omega\subseteq \R^{d_1\times d_2}}\{\loss_N(\Theta)+\regu_\lambda(\Theta)\},
\end{equation}
where $\loss_N:\R^{d_1\times d_2}\to \R$ is a loss function, $\lambda>0$ is a regularization parameter providing a tradeoff between model fitting and low-rankness, and $\regu_\lambda:\R^{d_1\times d_2}\to \R$ is a regularizer depending on $\lambda$ and imposing low-rankness of $\hat{\Theta}$.

In the following, both the loss function $\loss_N$ and the regularizer $\regu_\lambda$ can be nonconvex. We only require the differentiability $\loss_N$. Due to the nonconvexity, the feasible region is set to be a convex set as follows
\begin{equation}\label{eq-feasi}
\Omega:=\{\Theta\in \R^{d_1\times d_2}\big| \normm{\Theta}_*\leq \omega\},
\end{equation}
where $\omega>0$ must be chosen to guarantee that $\Theta^*$ is feasible, i.e., $\Theta^*\in \Omega$. Any matrix $\Theta\in \Omega$ also satisfies the side constraint $\|\Theta\|_*\leq \omega$. Then Weierstrass extreme value theorem ensures that global solutions of \eqref{eq-esti} always exist as long as the loss function and the regularizer are both continuous.  

The regularizer is set as a sum of a univariate function imposed on the singular value of a matrix, i.e., $\regu_\lambda(\cdot)=\sum_{j=1}^{d}p_\lambda(\sigma_j(\cdot))$, with $p_\lambda:\R \to \R$. Moreover, we assume that the univariate function $p_\lambda(\cdot)$ can be decomposed as $p_\lambda(\cdot)=q_\lambda(\cdot)+\lambda|\cdot|$, where $q_\lambda(\cdot)$ is a concave function and $|\cdot|$ is the absolute value function. Hence, it is easy to see that the regularizer can be decomposed as $\regu_\lambda(\cdot)=\Q_\lambda(\cdot)+\lambda\normm{\cdot}_*$, where $\Q_\lambda(\cdot)$ is the concave component defined by 
\begin{equation}\label{eq-qlambda}
\Q_\lambda(\cdot)=\sum_{j=1}^{d}q_\lambda(\sigma_j(\cdot)),
\end{equation} 
and $\normm{\cdot}_*$ is the nuclear norm function. 

\subsection{Regularity conditions}

In order to bound recovery errors for low-rank matrix estimation problems, researchers have introduced several types of regularity conditions, among which the RSC is one of the weakest conditions, and has been shown to be satisfied by a wide range of random matrices with overwhelming probability when the covariates are clean \cite{agarwal2012fast,negahban2011estimation}. 

However, it is still an open question whether or not a suitable form of RSC exists for errors-in-variables matrix regression. In this article, we provide a positive answer for this question by proposing a general type of RSC condition and verifying the condition for specific measurement error models under high-dimensional scaling. Some notations are needed first.

For an arbitrary matrix $\Delta\in \R^{d_1\times d_2}$, write the singular value decomposition as $\Delta=UDV^\top$, where $U\in \R^{d_1\times d}$ and $V^*\in \R^{d_2\times d}$ are orthonormal matrices with $d=\min\{d_1,d_2\}$. Let $\sigma(\Delta)$ be the vector formed by the singular values of $\Delta$ in decreasing order, and $T=\{1,2,\cdots,2r\}$ be an index set with $r=\text{rank}(\Theta^*)$. 
Define 
\begin{equation}\label{eq-thetaj}
\Delta_T=U\text{mat}(\sigma_T(\Delta))V^\top\quad  \mbox{and}\quad \Delta_{T^c}=U\text{mat}(\sigma_{T^c}(\Delta))V^\top,
\end{equation}
where $\text{mat}(\sigma_T(\Delta))$ represents the diagonal matrix with diagonal elements formed by the vector $\sigma_T(\Delta)$ and $\text{mat}(\sigma_T^c(\Delta))$ is given analogously. 
Then it is easy to see that $\Delta=\Delta_T+\Delta_{T^c}$ and $\inm{\Delta_T}{\Delta_{T^c}}=0$.

Define the cone set as  
\begin{equation}\label{eq-cone}
\Cc=\left\{\Delta\in \R^{d_1\times d_2}\big|\normm{\Delta_{T^c}}_*\leq 7\normm{\Delta_T}_*\right\}.
\end{equation}
Then the regularity conditions take two types of forms. One is the local strong convexity of the loss function around the true parameter $\Theta^*$; the other one is the restricted strong convexity of the loss function in a restricted cone set.

\begin{Definition}\label{asup-rsc1}
The function $\loss_N$ is said to satisfy the local strong convexity (\emph{LSC}) with parameters $\alpha_1>0$ and $\tau_1>0$ if
\begin{equation}\label{eq-rsc1}
\inm{\nabla\loss_N(\Theta^*+\Delta)-\nabla\loss_N(\Theta^*)}{\Delta}\geq \alpha_1\normm{\Delta}_\text{F}^2-\tau_1\normm{\Delta}_*^2,\quad \forall\  \Delta\in \R^{d_1\times d_2}.
\end{equation}
\end{Definition}

\begin{Definition}\label{asup-rsc2}
The function $\loss_N$ is said to satisfy the restricted strong convexity (\emph{RSC}) with parameters $\alpha_2>0$ if
\begin{equation}\label{eq-rsc2}
\loss_N(\Theta+\Delta)-\loss_N(\Theta)-\inm{\nabla\loss_N(\Theta)}{\Delta}\geq \alpha_2\normm{\Delta}_\text{F}^2,\quad \forall\  \Delta\in \Cc.
\end{equation}
\end{Definition}

Now we impose several regularity conditions on the nonconvex regularizer $\regu_\lambda$ in terms of the univariate functions $p_\lambda$ and $q_\lambda$.

\begin{Assumption}\mbox{}\par\label{asup-regu}
\begin{enumerate}[\rm(i)]
\item $p_\lambda$ satisfies $p_\lambda(0)=0$ and is symmetric around zero, that is, $p_\lambda(t)=p_\lambda(-t)$ for all $t\in \R$.
\item For $t>0$, the function $t\mapsto \frac{p_\lambda(t)}{t}$ is nonincreasing in $t$;
\item $p_\lambda$ is differentiable for all $t\neq 0$ and subdifferentiable at $t=0$, with $\lim\limits_{t\to 0^+}p'_\lambda(t)=\lambda$.
\item On the nonnegative real line, $p_\lambda$ is nondecreasing and concave, and there exists a constant $\nu$ that $p_\lambda$ satisfies $p'_\lambda(t)=0$ for all $t\geq \nu>0$.
\item For $t>t'$, there exists a positive constant $\mu\geq 0$ such that
    \begin{equation}\label{cond-qlambda}
    q'_\lambda(t)-q'_\lambda(t')\geq -\mu(t-t').
    \end{equation}
\item On the nonnegative real line, $q_\lambda$ is decreasing and $|q_\lambda'(t)|$ is upper bounded by $\lambda$, that is, $|q_\lambda'(t)|\leq \lambda$.
\end{enumerate}
\end{Assumption}
Note that condition (ii) implies that on the nonnegative line, the function $p_\lambda$ is subadditive. It is easy to check that the convex nuclear norm does not satisfy Assumption \ref{asup-regu} due to a violation of condition (iv). In fact, condition (iv) plays a key role in achieving a tighter recovery bound of nonconvex regularizers than that of the convex nuclear norm. Nonetheless, nonconvex regularizers such as SCAD and MCP are contained in this framework. 

Fixing $a>2$ and $b>0$, the function $p_\lambda$ for the SCAD regularizer is defined as 
\begin{equation*}
p_\lambda(t):=\left\{
\begin{array}{l}
\lambda|t|,\ \ \text{if}\ \  |t|\leq \lambda,\\
-\frac{t^2-2a\lambda|t|+\lambda^2}{2(a-1)},\ \  \text{if}\ \ \lambda<|t|\leq a\lambda,\\
\frac{(a+1)\lambda^2}{2},\ \ \text{if}\ \ |t|>a\lambda,
\end{array}
\right.
\end{equation*}
and the function $p_\lambda$ for the MCP regularizer is defined as
\begin{equation*}
p_\lambda(t):=\left\{
\begin{array}{l}
\lambda|t|-\frac{t^2}{2b},\ \ \text{if}\ \  |t|\leq b\lambda,\\
\frac{b\lambda^2}{2},\ \ \text{if}\ \ |t|>b\lambda.
\end{array}
\right.
\end{equation*}
It has been verified in \cite{gui2015towards} that the SCAD regularizer satisfies condition (iv) with $\nu=a\lambda$; while for MCP, condition (iv) is satisfied with $\nu=b\lambda$.
The two concave components are 
\begin{equation}\label{SCAD-q-2}
q_\lambda(t)=\left\{
\begin{array}{l}
0,\ \ \text{if}\ \  |t|\leq \lambda,\\
-\frac{t^2-2\lambda|t|+\lambda^2}{(2(a-1)},\ \  \text{if}\ \ \lambda<|t|\leq a\lambda,\\
\frac{(a+1)\lambda^2}{2}-\lambda|t|,\ \ \text{if}\ \ |t|>a\lambda,
\end{array}
\right.
\end{equation}
for SCAD with $\mu=\frac{1}{a-1}$, and
\begin{equation}\label{MCP-q-2}
q_\lambda(t)=\left\{
\begin{array}{l}
-\frac{t^2}{2b},\ \ \text{if}\ \  |t|\leq b\lambda,\\
\frac{b\lambda^2}{2}-\lambda|t|,\ \  \text{if}\ \ |t|> b\lambda,
\end{array}
\right.
\end{equation}
for MCP with $\mu=\frac{1}{b}$, respectively, for condition (v). 

At the end of this section, three technical lemmas are provided, showing some general properties of the nonconvex regularizer $\regu_\lambda$ and the concave component $\Q_\lambda$. The first lemma is from \cite[Theorem 1]{Rotfeld1967RemarksOT} and \cite[Theorem 1]{Yue2016API} about singular value inequalities with the proof omitted.

\begin{Lemma}\label{lem-tri}
Let $\Theta,\Theta'\in \R^{d_1\times d_2}$ be two given matrices and $d=\min\{d_1,d_2\}$. Let $f:\R_+\to \R_+$ be a concave increasing function satisfying $f(0)=0$. Then it holds that
\begin{align}
\sum_{j=1}^{d}f(\sigma_j(\Theta+\Theta'))&\leq \sum_{j=1}^{d}f(\sigma_j(\Theta))+\sum_{j=1}^{d}f(\sigma_j(\Theta')),\label{eq-tri+}\\
\sum_{j=1}^{d}f(\sigma_j(\Theta-\Theta'))&\geq \sum_{j=1}^{d}f(\sigma_j(\Theta))-\sum_{j=1}^{d}f(\sigma_j(\Theta')).\label{eq-tri-}
\end{align}
\end{Lemma}

\begin{Lemma}\label{lem-regu}
Suppose that $\regu_\lambda$ satisfy Assumption \ref{asup-regu}. Then for any $\Theta,\Theta'\in \R^{d_1\times d_2}$, we have that
\begin{align}
\regu_\lambda(\Theta+\Theta')&\leq \regu_\lambda(\Theta)+\regu_\lambda(\Theta'),\label{eq-regu+}\\
\regu_\lambda(\Theta-\Theta')&\geq \regu_\lambda(\Theta)-\regu_\lambda(\Theta').\label{eq-regu-}
\end{align}
\end{Lemma}
\begin{proof}
Since the singular values of a matrix is always nonnegative, the univariate function $p_\lambda$ actually  satisfies $p_\lambda: \R_+\to \R_+$. Then by Assumption \ref{asup-regu} (i) and (iv), Lemma \ref{lem-tri} is applicable to concluding that 
\begin{align*}
\sum_{j=1}^{d}p_\lambda(\sigma_j(\Theta+\Theta'))&\leq \sum_{j=1}^{d}p_\lambda(\sigma_j(\Theta))+\sum_{j=1}^{d}p_\lambda(\sigma_j(\Theta')),\\
\sum_{j=1}^{d}p_\lambda(\sigma_j(\Theta-\Theta'))&\geq \sum_{j=1}^{d}p_\lambda(\sigma_j(\Theta))-\sum_{j=1}^{d}p_\lambda(\sigma_j(\Theta')).
\end{align*}
The conclusion then follows directly from the definition of the regularizer $\regu_\lambda$.
\end{proof}

\begin{Lemma}\label{lem-qlambda}
Let $\Q_\lambda$ be defined in \eqref{eq-qlambda}. Then for any $\Theta,\Theta'\in \R^{d_1\times d_2}$, the following relations are true:
\begin{subequations}
\begin{align}
&\inm{\nabla\Q_\lambda(\Theta)-\nabla\Q_\lambda(\Theta')}{\Theta-\Theta'} \geq -\mu\normm{\Theta-\Theta'}_\emph{F}^2,\label{lem-qlambda-11}\\
&\inm{\nabla\Q_\lambda(\Theta)-\nabla\Q_\lambda(\Theta')}{\Theta-\Theta'}\leq 0, \label{lem-qlambda-12}\\
&\Q_\lambda(\Theta)\geq \Q_\lambda(\Theta')+\inm{\nabla\Q_\lambda(\Theta')}{\Theta-\Theta'}- \frac{\mu}{2}\normm{\Theta-\Theta'}_\emph{F}^2,\label{lem-qlambda-13}\\
&\Q_\lambda(\Theta)\leq \Q_\lambda(\Theta')+\inm{\nabla\Q_\lambda(\Theta')}{\Theta-\Theta'}.\label{lem-qlambda-14}
\end{align}
\end{subequations}
\end{Lemma}
\begin{proof}
For any matrices $\Theta,\Theta'\in \R^{d_1\times d_2}$, let $d=\min\{d_1,d_2\}$, and we use $\sigma,\sigma'$ to stand for the vectors consisting of singular values of $\Theta,\Theta'$ in decreasing order, respectively. 
Then one has the singular value decompositions for $\Theta,\Theta'$ as follows:
\begin{equation*}
\begin{aligned}
\Theta&=UDV^\top,\\
\Theta'&=U'D'{V'}^\top,
\end{aligned}
\end{equation*}
where $D,D'\in \R^{d\times d}$ are diagonal matrices with $D=\text{diag}(\sigma), D'=\text{diag}(\sigma')$. By Assumption \ref{asup-regu} (iv)-(v), we have for each pair of singular values of $\Theta,\Theta'$: ($\sigma_j,\sigma'_j$), $j=1,2,\cdots,d$, it holds that 
\begin{equation*}
-\mu(\sigma_j-\sigma'_j)^2\leq (q'_\lambda(\sigma_j)-q'_\lambda(\sigma'_j))(\sigma_j-\sigma'_j)\leq 0,
\end{equation*}
Then it follows from the definitions of $D,D'$ that 
\begin{equation*}
-\mu\normm{\Theta-\Theta'}_\text{F}^2\leq \inm{\nabla\Q_\lambda(UDV^\top)-\nabla\Q_\lambda(U'D'{V'}^\top)}{\Theta-\Theta'}\leq 0.
\end{equation*}
Thus \eqref{lem-qlambda-11} and \eqref{lem-qlambda-12} holds directly.
Combining \cite[Theorem 2.1.5 and Theorem 2.1.9]{nesterov2013introductory} and \eqref{lem-qlambda-11} and \eqref{lem-qlambda-12}, we have that
the convex function $-\Q_\lambda$ satisfies
\begin{align*}
-\Q_\lambda(\Theta)&\leq -\Q_\lambda(\Theta')+\inm{\nabla(-\Q_\lambda(\Theta'))}{\Theta-\Theta'}+ \frac{\mu}{2}\normm{\Theta-\Theta'}_\text{F}^2,\\
-\Q_\lambda(\Theta)&\geq -\Q_\lambda(\Theta')+\inm{\nabla(-\Q_\lambda(\Theta'))}{\Theta-\Theta'},
\end{align*}
which respectively implies that the function $\Q_\lambda$ satisfies \eqref{lem-qlambda-13} and \eqref{lem-qlambda-14}.
The proof is complete.
\end{proof}

\section{Main results}\label{sec-main}

In this section, we establish the main result on the recovery bound for any stationary point of the general nonconvex estimator \eqref{eq-esti}. The result is deterministic in nature, and probabilistic consequences for the errors-in-variables matrix regression model are given in the next section. 

Before we proceed, some additional notations are needed. Let $\obj(\Theta)=\loss_N(\Theta)+\regu_\lambda(\Theta)$ represent the objective function to be minimized. Recall that the regularizer can be decomposed as $\regu_\lambda(\Theta)=\Q_\lambda(\Theta)+\lambda\normm{\Theta}_*$. Then it holds that $\obj(\Theta)=\loss_N(\Theta)+\Q_\lambda(\Theta)+\lambda\normm{\Theta}_*$. Denote $\tilde{\loss}_N(\Theta)=\loss_N(\Theta)+\Q_\lambda(\Theta)$, and it follows that $\obj(\Theta)=\tilde{\loss}_N(\Theta)+\lambda\normm{\Theta}_*$. In this way, one sees that the objective function is decomposed into a differentiable but nonconvex function and a nonsmooth but convex function.

Let $d=\min\{d_1,d_2\}$. Consider the singular value
decomposition of the parameter matrix $\Theta^*=U^*D^*{V^*}^\top$, where $U^*\in \R^{d_1\times d}$ and $V^*\in \R^{d_2\times d}$
are orthonormal, and $D\in \R^{d\times d}$ is diagonal. By the low-rank assumption \eqref{eq-rank}, one has that there are only $r$ nonzero elements on the diagonal of $D$. Specifically, $\sigma(\Theta^*)={(\sigma_1(\Theta^*),\sigma_2(\Theta^*),\cdots,\sigma_r(\Theta^*),0,\cdots,0)}^\top\in \R^d$ is the vector formed by the diagonal elements of $D$, with $\sigma_1(\Theta^*)\geq \sigma_2(\Theta^*)\geq \cdots \sigma_r(\Theta^*)>0$.

For any index set $S\subseteq \{1,2,\cdots,d\}$, we define the following two subspaces of $\R^{d_1\times d_2}$ associated with $\Theta^*$ as:
\begin{subequations}\label{eq-sub}
\begin{align}
\Aa_S(U^*,V^*)&:=\{\Delta\in \R^{d_1\times d_2}\big|\text{row}(\Delta)\subseteq \text{col}(V_S^*), \text{col}(\Delta)\subseteq \text{col}(U_S^*)\},\label{eq-sub1}\\
\Bb_S(U^*,V^*)&:=\{\Delta\in \R^{d_1\times d_2}\big|\text{row}(\Delta)\perp \text{col}(V_S^*), \text{col}(\Delta)\perp \text{col}(U_S^*)\},\label{eq-sub2}
\end{align}
\end{subequations}
where $\text{row}(\Delta)\in \R^{d_2}$ and $\text{col}(\Delta)\in \R^{d_1}$ respectively stand for the row space and column space of the matrix $\Delta$, and $V_S^*$ and $U_S^*$ respectively represent the matrix formed by the rows of $V^*$ and $U^*$ indexed by the set $S$. 
When matrices $(U^*,V^*)$ are known from the context, we use the shorthand notation $\Aa_S$ and $\Bb_S$ instead. Two projection operators onto the subspaces $\Aa_S$ and $\Bb_S$ are defined as follows:
\begin{subequations}\label{eq-proj}
\begin{align}
\Pi_{\Aa_S}(\Theta)&=U_S^*{U_S^*}^\top\Theta V_S^*{V_S^*}^\top,\label{eq-proj1}\\
\Pi_{\Bb_S}(\Theta)&=(\I_{d_1}-U_S^*{U_S^*}^\top)\Theta(\I_{d_2}-V_S^*{V_S^*}^\top).\label{eq-proj2}
\end{align}
\end{subequations}

Similar definitions of $\Aa_S$ and $\Bb_S$ have been introduced in \cite{agarwal2012fast,negahban2011estimation} to study low-rank recovery problems with clean covariates. Particularly, let $J$ denote the index set corresponding to the nonzero elements of $\sigma(\Theta^*)$, i.e., 
\begin{equation}\label{eq-J}
J=\{1,2,\cdots,r\}.
\end{equation}
Then the nuclear norm is decomposable with respect to $\Aa_J$ and $\Bb_J$, that is, $\normm{\Theta+\Theta'}_*=\normm{\Theta}_*+\normm{\Theta'}_*$ holds for any pair of matrices $\Theta\in \Aa_J$ and $\Theta'\in \Bb_J$. Moreover, recall the constant in Assumption \ref{asup-regu} (iv) and define two index set corresponding to the larger and smaller singular values of $\Theta^*$ as follows:
\begin{equation}\label{eq-muj}
J_1=\{j\big|\sigma_j(\Theta^*)\geq \mu\}\quad  \mbox{and}\quad J_2=\{j\big|0<\sigma_j(\Theta^*)<\mu\},
\end{equation}
with $|J_1|=r_1$ and $|J_2|=r_2$. It is easy to see that $r_1+r_2=r$ and that $\Pi_{\Aa_J}(\Delta)=\Pi_{\Aa_{J_1}}(\Delta)+\Pi_{\Aa_{J_2}}(\Delta)$ for any matrix $\Delta\in \R^{d_1\times d_2}$.


We now state a useful technical lemma that helps us decompose the error matrix $\Delta:=\Theta-\Theta^*$, where $\Theta^*$ is the true parameter matrix and $\Theta$ is arbitrary, as the sum of two matrices $\Delta'$ and $\Delta''$ such that the
rank of $\Delta'$ is not too large. In addition, the difference between $\regu_\lambda(\Theta^*)$ and $\regu_\lambda(\Theta)$ can be bounded from above in terms of the nuclear norms. 



\begin{Lemma}\label{lem-decom}
Let $\Theta\in \R^{d_1\times d_2}$ be an arbitrary matrix. Let $T=\{1,2,\cdots,2r\}$, and $\Delta_T$ and $\Delta_{T^c}$ be given in \eqref{eq-thetaj}. Then the following conclusions hold:\\
\rm (i) there exists a decomposition $\Delta=\Delta'+\Delta''$ such that the matrix $\Delta'$ with $\text{rank}(\Delta')\leq 2r$;\\
\rm (ii) $\regu_\lambda(\Theta^*)-\regu_\lambda(\Theta)\leq \lambda(\normm{\Delta_T}_*-\normm{\Delta_{T^c}}_*)$.
\end{Lemma}
\begin{proof}
(i) The first part of this lemma is proved in  \cite[Lemma 3.4]{recht2010guaranteed}, and we here provide the proof for completeness. Write the singular value decomposition of $\Theta^*$ as $\Theta^*=U^*D^*{V^*}^\top$, where $U^*\in \R^{d_1\times d_1}$ and $V^*\in \R^{d_2\times d_2}$ are orthogonal matrices, and $D^*\in\R^{d_1\times d_2}$ is the matrix formed by the singular values of $\Theta^*$. Define the matrix $\Xi={U^*}^\top\Delta V^*\in \R^{d_1\times d_2}$, and partition $\Xi$ in block form as follows
      \[
      \Xi:=\begin{pmatrix}
        \Xi_{11} & \Xi_{12} \\
        \Xi_{21} & \Xi_{22}
        \end{pmatrix},\ \text{where}\ \Xi_{11}\in \R^{r\times r}\ \text{and}\  \Xi_{22}\in \R^{(m_1-r)\times (m_2-r)}.
      \]
      Set the matrices as
      \[
      \Delta':=U^*\begin{pmatrix}
        \textbf{0} & \textbf{0} \\
        \textbf{0} & \Xi_{22}
        \end{pmatrix}{V^*}^\top\ \text{and}\  \Delta'':=\Delta-\Delta'.
      \]
      Then the rank of $\Delta'$ is upper bounded as
      \begin{equation*}
        \text{rank}(\Delta')=\text{rank}\begin{pmatrix}
        \Xi_{11} & \Xi_{12} \\
        \Xi_{21} & \textbf{0}
        \end{pmatrix}\leq
        \text{rank}\begin{pmatrix}
        \Xi_{11} & \Xi_{12} \\
        \textbf{0} & \textbf{0}
        \end{pmatrix}+
        \text{rank}\begin{pmatrix}
        \Xi_{11} & \textbf{0} \\
        \Xi_{21} & \textbf{0}
        \end{pmatrix}\leq 2r,
      \end{equation*}
      which established Lemma \ref{lem-decom}(i).\\
(ii) 
It follows from the constructions of $\Delta'$ and $\Delta''$
that $\sigma(\Delta')+\sigma(\Delta'')=\sigma(\Delta)$ with $|\text{supp}(\sigma(\Delta'))|\leq 2r$ and $\inm{\Delta'}{\Delta''}=0$.
    On the other hand, recall the definition of the set $J$. Note that the decomposition $\Theta^*=\Pi_{\Aa_J}(\Theta^*)+\Pi_{\Bb_J}(\Theta^*)=\Pi_{\Aa_J}(\Theta^*)$ holds due to the fact that $\Pi_{\Bb_J}(\Theta^*)=\textbf{0}$. This equality, together with Lemma \ref{lem-regu}, implies that
      \begin{equation}\label{eq-lem4-2}
      \begin{aligned}
       \regu_\lambda(\Theta)&=\regu_\lambda(\Pi_{\Aa_J}(\Theta^*)+\Delta''+\Delta'))\\
       &=\regu_\lambda[(\Pi_{\Aa_J}(\Theta^*)+\Delta'')-(-\Delta')]\\
       &\geq \regu_\lambda(\Pi_{\Aa_J}(\Theta^*)+\Delta'')-\regu_\lambda(-\Delta')\\
       &\geq \regu_\lambda(\Pi_{\Aa_J}(\Theta^*))+\regu_\lambda(\Delta'')-\regu_\lambda(\Delta'),
      \end{aligned}
      \end{equation}
where the last inequality is from Assumption \ref{asup-regu} (i). 
      Consequently, we have
      \begin{equation}\label{eq-lem4-3}
      \begin{aligned}
      \regu_\lambda(\Theta^*)-\regu_\lambda(\Theta)
      &\leq
      \regu_\lambda(\Theta^*)-\regu_\lambda(\Pi_{\Aa_J}(\Theta^*))-\regu_\lambda(\Delta'')+\regu_\lambda(\Delta')\\
      &\leq
      \regu_\lambda(\Delta')-\regu_\lambda(\Delta'')\\
      &\leq
      \regu_\lambda(\Delta_T)-\regu_\lambda(\Delta_{T^c}),
      \end{aligned}
      \end{equation}
      where the last inequality is from the definition of the set $T$ and Assumption \ref{asup-regu} (iv).
Then it follows from \cite[Lemma 6]{loh2013local} that 
$$\sum_{i=1}^dp_\lambda(\sigma_i(\Delta_T))-\sum_{i=1}^dp_\lambda(\sigma_i(\Delta_{T^c}))\leq \lambda\left(\sum_{i=1}^d\sigma_i(\Delta_T)-\sum_{i=1}^d\sigma_j(\Delta_{T^c})\right).$$
Combining this inequality with \eqref{eq-lem4-3} and the definition of $\regu_\lambda$, we arrive at the conclusion. The proof is complete.
\end{proof}

Recall the feasible region $\Omega$ given in \eqref{eq-feasi}. The next lemma shows that the error matrix $\tilde{\Delta}:=\tilde{\Theta}-\Theta^*$ belongs to a certain set, where $\tilde{\Theta}\in \Omega$ is an arbitrary stationary point of the optimization problem \eqref{eq-esti} satisfing the first-order necessary condition:
\begin{equation}\label{1st-cond}
\inm{\nabla\loss_N(\tilde{\Theta})+\nabla\regu_\lambda(\tilde{\Theta})}{\Theta-\tilde{\Theta}}\geq 0,\quad \text{for all}\ \Theta\in \Omega.
\end{equation}

\begin{Lemma}\label{lem-cone}
Let $r, \omega$ be positive numbers such that $\Theta^*\in \Omega$ and satisfies \eqref{eq-rank}. Let $\tilde{\Theta}$ be a stationary point of the optimization problem \eqref{eq-esti}. Let $T=\{1,2,\cdots,2r\}$, and $\tilde{\Delta}_T$ and $\tilde{\Delta}_{T^c}$ be given in \eqref{eq-thetaj}. Suppose that the nonconvex regularizer $\regu_\lambda$ satisfies Assumption \ref{asup-regu}, and that the empirical loss function $\loss_N$ satisfies the \emph{LSC} condition (cf. \eqref{eq-rsc1}) with $\alpha_1>\mu$.
Assume that $(\lambda, \omega)$ are chosen to satisfy $\lambda\geq 2\max\{\normm{\nabla\loss_N(\Theta^*)}_\text{op},4\omega\tau_1\}$, then one has that 
\begin{equation}\label{eq-lemcone}
\normm{\tilde{\Delta}_{T^c}}_*\leq 7\normm{\tilde{\Delta}_T}_*.
\end{equation}
\end{Lemma}
\begin{proof}
By the RSC condition \eqref{eq-rsc1}, one has that
\begin{equation}\label{eq-lem3-1}
\inm{\nabla\loss_N(\tilde{\Theta})-\nabla\loss_N(\Theta^*)}{\tilde{\Delta}}\geq \alpha_1\normm{\tilde{\Delta}}_\text{F}^2-\tau_1\normm{\tilde{\Delta}}_*^2.
\end{equation}
On the other hand, it follows from \eqref{lem-qlambda-11} and  \eqref{lem-qlambda-14} in Lemma \ref{lem-qlambda} that
\begin{equation*}
\begin{aligned}
\inm{\nabla\regu_\lambda(\tilde{\Theta})}{\Theta^*-\tilde{\Theta}}
&= \inm{\nabla\Q_\lambda(\tilde{\Theta})+\lambda\tilde{G}}{\Theta^*-\tilde{\Theta}}\\
&\leq \inm{\nabla\Q_\lambda(\Theta^*)}{\Theta^*-\tilde{\Theta}}+\mu\normm{\Theta^*-\tilde{\Theta}}_\text{F}^2+\inm{\lambda\tilde{G}}{\Theta^*-\tilde{\Theta}}\\
&\leq \Q_\lambda(\Theta^*)-\Q_\lambda(\tilde{\Theta})+\mu\normm{\Theta^*-\tilde{\Theta}}_\text{F}^2+\inm{\lambda\tilde{G}}{\Theta^*-\tilde{\Theta}},
\end{aligned}
\end{equation*}
where $\tilde{G}\in \partial \normm{\tilde{\Theta}}_*$.
Moreover, since the function $\normm{\cdot}_*$ is convex, one has that
\begin{equation*}
\normm{\Theta^*}_*-\normm{\tilde{\Theta}}_*\geq \inm{\tilde{G}}{\Theta^*-\tilde{\Theta}}.
\end{equation*}
This, together with the former inequality, implies that
\begin{equation}\label{eq-lem3-2}
\inm{\nabla\regu_\lambda(\tilde{\Theta})}{\Theta^*-\tilde{\Theta}}\leq 
\regu_\lambda(\Theta^*)-\regu_\lambda(\tilde{\Theta})+\mu\normm{\tilde{\Delta}}_\text{F}^2.
\end{equation}
Then combining \eqref{eq-lem3-1}, \eqref{eq-lem3-2} and \eqref{1st-cond} (with $\Theta^*$ in place of $\Theta$), we obtain that
\begin{equation}\label{eq-lem3-3}
\begin{aligned}
\alpha_1\normm{\tilde{\Delta}}_\text{F}^2-\tau_1\normm{\tilde{\Delta}}_*^2
&\leq -\inm{\nabla\loss_N(\Theta^*)}{\tilde{\Delta}}+\regu_\lambda(\Theta^*)-\regu_\lambda(\tilde{\Theta})+\mu\normm{\tilde{\Delta}}_\text{F}^2\\
&\leq  \normm{\nabla\loss_N(\Theta^*)}_{\text{op}}\normm{\tilde{\Delta}}_*+\regu_\lambda(\Theta^*)-\regu_\lambda(\tilde{\Theta})+\mu\normm{\tilde{\Delta}}_\text{F}^2\\
&\leq  \frac{\lambda}{2}\normm{\tilde{\Delta}}_*+\regu_\lambda(\Theta^*)-\regu_\lambda(\tilde{\Theta})+\mu\normm{\tilde{\Delta}}_\text{F}^2\\
\end{aligned}
\end{equation}
where the second inequality is from H{\"o}lder's inequality, and the last inequality is from the assumption that $\lambda\geq 2\normm{\nabla\loss_N(\Theta^*)}_{\text{op}}$. 
It then follows from Lemma \ref{lem-decom} (ii) and noting the fact that $\normm{\tilde{\Delta}}_*\leq \normm{\Theta^*}_*+\normm{\tilde{\Theta}}_*\leq 2\omega$, one has from \eqref{eq-lem3-3} that 
\begin{equation}\label{eq-lem3-4}
\begin{aligned}
(\alpha_1-\mu)\normm{\tilde{\Delta}}_\text{F}^2&\leq \frac{\lambda}{2}\normm{\tilde{\Delta}}_*+2\omega\tau_1\normm{\tilde{\Delta}}+\lambda(\normm{\tilde{\Delta}_T}_*-\normm{\tilde{\Delta}_{T^c}}_*)\\
&\leq \frac{3}{4}\lambda\normm{\tilde{\Delta}}_*+\lambda(\normm{\tilde{\Delta}_T}_*-\normm{\tilde{\Delta}_{T^c}}_*)\\
&\leq \frac{3}{4}\lambda(\normm{\tilde{\Delta}_T}_*+\normm{\tilde{\Delta}_{T^c}}_*)+\lambda(\normm{\tilde{\Delta}_T}_*-\normm{\tilde{\Delta}_{T^c}}_*),
\end{aligned}
\end{equation}
where the second inequality is due to the assumption that $\lambda\geq 8\omega\tau_1$ and the last inequality is from triangle inequality. Since $\alpha_1>\mu$ by assumption, one has by the former inequality that \eqref{eq-lemcone} holds.
\end{proof}

We are now ready to provide the recovery bound for any stationary point $\tilde{\Theta}\in \Omega$ of the nonconvex optimization problem \eqref{eq-esti}.

\begin{Theorem}\label{thm-stat}
Let $r, \omega$ be positive numbers such that $\Theta^*\in \Omega$ and satisfies \eqref{eq-rank}. Let $\tilde{\Theta}$ be a stationary point of the optimization problem \eqref{eq-esti}. Suppose that the nonconvex regularizer $\regu_\lambda$ satisfies Assumption \ref{asup-regu}, and that the empirical loss function $\loss_N$ satisfies the \emph{LSC} and \emph{RSC} conditions (cf. \eqref{eq-rsc1} and \eqref{eq-rsc2}) with $\min\{\alpha_1,2\alpha_2\}>\mu$.
Assume that $(\lambda, \omega)$ are chosen to satisfy
\begin{equation}\label{eq-lambda-sta}
\lambda\geq 2\max\{\normm{\nabla\loss_N(\Theta^*)}_\emph{op},4\omega\tau_1\},
\end{equation}
then we have that 
\begin{equation}\label{eq-l2bound}
\normm{\hat{\Theta}-\Theta^*}_\text{F}\leq \frac{\sqrt{r_1}}{2\alpha_2-\mu}\normm{\Pi_{\Aa_{J_1}}(\nabla\loss_N(\Theta^*))}_\emph{op}+\frac{5\sqrt{r_2}}{2(2\alpha_2-\mu)}\lambda,
\end{equation}
\end{Theorem}
\begin{proof}
Set $\tilde{\Delta}:=\tilde{\Theta}-\Theta^*$. It follows from Lemma \ref{lem-cone} that $\tilde{\Delta}\in \Cc$, and thus the RSC condition \eqref{eq-rsc2} is applicable for $\tilde{\Delta}$ and $\Theta^*$.
Combining the RSC condition \eqref{eq-rsc2} and Lemma \ref{lem-qlambda} (cf. \eqref{lem-qlambda-13}) that
\begin{align}
\tilde{\loss}_N(\tilde{\Theta})-\tilde{\loss}_N(\Theta^*)-\inm{\nabla\tilde{\loss}_N(\Theta^*)}{\tilde{\Delta}}&\geq \frac{2\alpha_2-\mu}{2}\normm{\tilde{\Delta}}_\text{F}^2,\label{eq-thm1-1}\\
\tilde{\loss}_N(\Theta^*)-\tilde{\loss}_N(\tilde{\Theta})+\inm{\nabla\tilde{\loss}_N(\tilde{\Theta})}{\tilde{\Delta}}&\geq \frac{2\alpha_2-\mu}{2}\normm{\tilde{\Delta}}_\text{F}^2.\label{eq-thm1-2}
\end{align}
Then by the convexity of $\normm{\cdot}_*$, one has that 
\begin{equation}\label{eq-thm1-3}
\lambda\normm{\tilde{\Theta}}_*\geq \lambda\normm{\Theta^*}_*+\lambda\inm{\tilde{\Delta}}{G^*},
\end{equation}
\begin{equation}\label{eq-thm1-4}
\lambda\normm{\Theta^*}_*\geq \lambda\normm{\tilde{\Theta}}_*-\lambda\inm{\tilde{\Delta}}{\tilde{G}},
\end{equation}
where $G^*\in \partial\normm{\Theta^*}_*$ and $\tilde{G}\in \partial\normm{\tilde{\Theta}}_*$. Combing \eqref{eq-thm1-1}-\eqref{eq-thm1-4}, we have that
\begin{equation}\label{eq-thm1-5}
(2\alpha_2-\mu)\normm{\tilde{\Delta}}_\text{F}^2\leq \inm{\nabla\tilde{\loss}_N(\tilde{\Theta})+\lambda \tilde{G}}{\tilde{\Delta}}-\inm{\nabla\tilde{\loss}_N(\Theta^*)+\lambda G^*}{\tilde{\Delta}}.
\end{equation}
Since $\tilde{\Theta}$ is a stationary point of \eqref{eq-esti}, one has by \eqref{1st-cond} that 
\begin{equation*}
\inm{\nabla\loss_N(\tilde{\Theta})+\nabla\regu_\lambda(\tilde{\Theta})}{\Theta^*-\tilde{\Theta}}\geq 0,
\end{equation*}
and thus 
\begin{equation*}
\inm{\nabla\tilde{\loss}_N(\tilde{\Theta})+\lambda\tilde{G}}{\tilde{\Delta}}\leq 0,
\end{equation*}
by the definitions of $\tilde{\loss}_N$ and $\regu_\lambda$. Combining the former inequality with \eqref{eq-thm1-5}, we arrive at that
\begin{equation*}
(2\alpha_2-\mu)\normm{\tilde{\Delta}}_\text{F}^2\leq \inm{\nabla\tilde{\loss}_N(\Theta^*)+\lambda G^*}{-\tilde{\Delta}}.
\end{equation*}
Recall the definitions of two subspaces given in \eqref{eq-sub} and set $S=J$ (cf. \eqref{eq-J}). The former inequality turns to 
\begin{equation}\label{eq-thm1-main}
(2\alpha_2-\mu)\normm{\tilde{\Delta}}_\text{F}^2\leq \inm{\Pi_{\Aa_J}(\nabla\tilde{\loss}_N(\Theta^*)+\lambda G^*)}{-\tilde{\Delta}}+\inm{\Pi_{\Bb_J}(\nabla\tilde{\loss}_N(\Theta^*)+\lambda G^*)}{-\tilde{\Delta}}.
\end{equation}
Moreover, recall the constant in Assumption \ref{asup-regu} (iv) and define the following two index set 
\begin{equation}\label{eq-muj}
J_1=\{j\big|\sigma_j(\Theta^*)\geq \mu\}\quad  \mbox{and}\quad J_2=\{j\big|0<\sigma_j(\Theta^*)<\mu\}.
\end{equation}
Then we have that 
\begin{equation}\label{eq-Pij}
\inm{\Pi_{\Aa_J}(\nabla\tilde{\loss}_N(\Theta^*)+\lambda G^*)}{-\tilde{\Delta}}=\inm{\Pi_{\Aa_{J_1}}(\nabla\tilde{\loss}_N(\Theta^*)+\lambda G^*)}{-\tilde{\Delta}}+\inm{\Pi_{\Aa_{J_2}}(\nabla\tilde{\loss}_N(\Theta^*)+\lambda G^*)}{-\tilde{\Delta}}
\end{equation}
Now set $H^*:=-\lambda^{-1}\Pi_{\Bb_J}(\nabla\loss_N(\Theta^*))$, and it holds that 
\begin{equation}\label{eq-partial}
G^*=U^*{V^*}^\top+H^*\in \partial\normm{\Theta^*}_*.
\end{equation}
In fact, since $\normm{\Pi_{\Bb_J}(\nabla\loss_N(\Theta^*))}_\text{op}\leq \normm{\nabla\loss_N(\Theta^*)}_\text{op}\leq \lambda$ by assumption, one has that $\normm{H^*}_\text{op}\leq 1$ and $H^*\in \Bb_J$ such that $H^*$ and $\Theta^*$ having orthogonal row/column space. 
With this specific choice of the subgradient of $\normm{\Theta^*}_*$, we shall bound the two terms on the right-hand side of \eqref{eq-Pij} and the last term of \eqref{eq-thm1-main}. First for an index set $S\subseteq \{1,2,\cdots,d\}$, let $Q'_\lambda(\sigma_S(\Theta^*))=(q'_\lambda(\sigma_j(\Theta^*)))^\top\ (j\in S)$ for notational simplicity. 


(i) 
Since $\nabla\regu_\lambda(\Theta^*)=\nabla\Q_\lambda(\Theta^*)+\lambda G^*=\nabla\Q_\lambda(\Theta^*)+\lambda(U^*{V^*}^\top+H^*)$ and $H^*\perp \Aa_{J_1}$, we have the following projection onto the subspace $\Aa_{J_1}$ as 
\begin{equation*}
\begin{aligned}
\Pi_{\Aa_{J_1}}(\nabla\regu_\lambda(\Theta^*))
&=U_{J_1}\text{mat}(Q'_\lambda(\sigma_{J_1}(\Theta^*)))V_{J_1}^\top+\lambda U_{J_1}V_{J_1}^\top\\
&=U_{J_1}[\text{mat}(Q'_\lambda(\sigma_{J_1}(\Theta^*)))+\lambda I_{J_1}]V_{J_1}^\top.
\end{aligned}
\end{equation*}
Then $\text{mat}(Q'_\lambda(\sigma_{J_1}(\Theta^*)))+\lambda I_{J_1}$ is a diagonal matrix with its diagonal elements equal to $q'_\lambda(\sigma_j(\Theta^*))+\lambda=p'_\lambda(\sigma_j(\Theta^*))=0$, where the last equality follows from Assumption \ref{asup-regu} (iv) and that $\sigma_j(\Theta^*)\geq \nu$ for $j\in J_1$. Hence one has that
$\Pi_{\Aa_{J_1}}(\nabla\regu_\lambda(\Theta^*))=0$, and thus
\begin{equation}\label{eq-thm1-8}
\begin{aligned}
\inm{\Pi_{\Aa_{J_1}}(\nabla\tilde{\loss}_N(\Theta^*)+\lambda G^*)}{-\tilde{\Delta}}&=\inm{\Pi_{\Aa_{J_1}}(\nabla\loss_N(\Theta^*))}{\Pi_{\Aa_{J_1}}(-\tilde{\Delta})}\\
&\leq \normm{\Pi_{\Aa_{J_1}}(\nabla\loss_N(\Theta^*))}_\text{op}\normm{\Pi_{\Aa_{J_1}}(-\tilde{\Delta})}_*\\
&\leq \sqrt{|J_1|}\normm{\Pi_{\Aa_{J_1}}(\nabla\loss_N(\Theta^*))}_\text{op}\normm{\Pi_{\Aa_{J_1}}(-\tilde{\Delta})}_\text{F}\\
&\leq \sqrt{|J_1|}\normm{\Pi_{\Aa_{J_1}}(\nabla\loss_N(\Theta^*))}_\text{op}\normm{\tilde{\Delta}}_\text{F},
\end{aligned}
\end{equation}
where the second inequality is from H{\"o}lder's inequality, and the third inequality is due to the fact that the rank of $\Pi_{\Aa_{J_1}}(\tilde{\Delta})$ is less than $|J_1|$. 

(ii) On the other hand, projecting $\nabla \Q_\lambda$ onto the subspace $\Aa_{J_2}$ yields that 
$$\Pi_{\Aa_{J_2}}(\nabla\Q_\lambda(\Theta^*))=U_{J_2}\text{mat}(Q'_\lambda(\sigma_{J_2}(\Theta^*)))(V_{J_2})^\top,$$
in which $\text{mat}(Q'_\lambda(\sigma_{J_2}(\Theta^*)))$ is a diagonal matrix with its diagonal elements satisfying $q'_\lambda(\sigma_j(\Theta^*))\leq \lambda$ due to Assumption \ref{asup-regu} (vi).
Hence one has that \begin{equation}\label{eq-thm1-9}
\normm{\Pi_{\Aa_{J_2}}(\nabla\Q_\lambda(\Theta^*))}_\text{op}\leq \lambda.
\end{equation} 
Then projecting $\lambda G^*$ onto the subspace $\Aa_{J_2}$ yields that
\begin{equation}\label{eq-thm1-10}
\normm{\Pi_{\Aa_{J_2}}(\lambda G^*)}_\text{op}\leq \normm{\lambda U^*{V^*}^\top}_\text{op}
\leq \lambda,
\end{equation} 
where the second equality is due to the facts that $\normm{U^*{V^*}^\top}_{\text{op}}=1$ and that $\normm{H^*}_\text{op}\leq 1$. Therefore, combing \eqref{eq-thm1-9} and \eqref{eq-thm1-10} with the fact that $\normm{\Pi_{\Aa_{J_2}}(\nabla\loss_N(\Theta^*))}_{\text{op}}\leq \normm{\nabla\loss_N(\Theta^*)}_{\text{op}}\leq \lambda/2$, we obtain that
\begin{equation}\label{eq-thm1-11}
\begin{aligned}
\inm{\Pi_{\Aa_{J_2}}(\nabla\tilde{\loss}_N(\Theta^*)+\lambda G^*)}{-\tilde{\Delta}}
&\leq \normm{\Pi_{\Aa_{J_2}}(\nabla\loss_N(\Theta^*)+\nabla\Q_\lambda(\Theta^*)+\lambda G^*}_{\text{op}}\normm{\Pi_{\Aa_{J_2}}(-\tilde{\Delta})}_*\\
&\leq \frac{5}{2}\lambda\normm{\Pi_{\Aa_{J_2}}(\tilde{\Delta})}_*\leq \frac{5}{2}\sqrt{|J_2|}\lambda\normm{\tilde{\Delta}}_\text{F}.
\end{aligned}
\end{equation}

(iii) By the definition of \eqref{eq-proj2}, one has that 
\begin{equation*}
\Pi_{\Bb_J}(\nabla\Q_\lambda(\Theta^*))=(\I_{d_1}-U^*{U^*}^\top)U^*\text{mat}(Q'_\lambda(\sigma(\Theta^*))){V^*}^\top(\I_{d_2}-V^*{V^*}^\top)=\textbf{0}.
\end{equation*}
Then it holds that 
\begin{equation}\label{eq-Pibj}
\begin{aligned}
\normm{\Pi_{\Bb_J}(\nabla\loss_N(\Theta^*)+\nabla\Q_\lambda(\Theta^*)+\lambda G^*)}_{\text{op}}&=\normm{\Pi_{\Bb_J}(\nabla\loss_N(\Theta^*)+\lambda (U^*{V^*}^\top+H^*))}_{\text{op}}\\
&=\normm{\Pi_{\Bb_J}(\nabla\loss_N(\Theta^*))-\Pi_{\Bb_J}(\nabla\loss_N(\Theta^*))}_{\text{op}}=0.
\end{aligned}
\end{equation}
where the first and second equalities are from the definitions of $G^*$ and $H^*$.
Combing \eqref{eq-Pibj} with H{\"o}lder's inequality that
\begin{equation}\label{eq-thm1-12}
\inm{\Pi_{\Bb_J}(\nabla\tilde{\loss}_N(\Theta^*)+\lambda G^*)}{-\tilde{\Delta}}=0.
\end{equation}

Combining \eqref{eq-thm1-8}, \eqref{eq-thm1-11} and \eqref{eq-thm1-12} with \eqref{eq-thm1-main} and \eqref{eq-Pij}, we arrive at that
\begin{equation*}
(2\alpha_2-\mu)\normm{\tilde{\Delta}}_\text{F}^2\leq \sqrt{|J_1|}\normm{\Pi_{\Aa_{J_1}}(\nabla\loss_N(\Theta^*))}_\text{op}\normm{\tilde{\Delta}}_\text{F}+\frac{5}{2}\sqrt{|J_2|}\lambda\normm{\tilde{\Delta}}_\text{F}.
\end{equation*}
Recall that $|J_1|=r_1$ and $|J_2|=r_2$, and we arrive at the conclusion.
\end{proof}

\begin{Remark}\label{rmk-stat}

{\rm (i)} Theorem \ref{thm-stat} provides a unified framework for low-rank matrix recovery via nonconvex optimization, and establishes the recovery bound for any stationary point of the general nonconvex estimator \eqref{eq-esti} scaling as $\normm{\hat{\Theta}-\Theta^*}_\text{F}=O(\sqrt{r_1}\normm{\Pi_{\Aa_{J_1}}(\nabla\loss_N(\Theta^*))}_\emph{op}+\sqrt{r_2}\lambda)$, with $r_1$ corresponding to the number of larger singular values of the true parameter $\Theta^*$ and $r_2$ referring to the number of smaller singular values. As long as $r_1>0$, this recovery bound is tighter than that of the nuclear norm regularization method in \cite{negahban2011estimation,Li2024LowrankME} thanks to the fact that $\normm{\Pi_{\Aa_{J_1}}(\nabla\loss_N(\Theta^*))}_\emph{op}$ is always of smaller scale than $\normm{\nabla\loss_N(\Theta^*)}_\emph{op}$, which will be clarified in the next section for specific errors-in-variables matrix regression. Furthermore, when all the singular values of $\Theta^*$ are all larger than $\nu$, i.e., $\sigma_r(\Theta^*)>\nu$, this general nonconvex estimator achieves the best-case  recovery bound scaling as $\normm{\hat{\Theta}-\Theta^*}_\text{F}=O(\sqrt{r}\normm{\Pi_{\Aa_{J_1}}(\nabla\loss_N(\Theta^*))}_\emph{op})$. 

{\rm (ii)} Note that due to the nonconvexity of the optimization problem \eqref{eq-esti}, it is always difficult to obtain a global solution. Nonetheless, our result here is not affected by this problem since the recovery bound is provided for all stationary points. Therefore, Theorem \ref{thm-stat} does not rely on any specific algorithm, implying that any numerical procedure can stably recover the true low-rank matrix as long as it is ensured to converge to a stationary point. This advantage is advantage is due to the LSC condition which requires that the loss function is strongly convex around a small neighbourhood of the true parameter. \cite{gui2015towards} also considered the problem of low-rank matrix recovery with nonconvex regularizers while the loss function is still convex. However, the result there are only applicable for a global solution whereas Theorem \ref{thm-stat} is much stronger and tighter for any stationary point and covers more general nonconvex regularizers beyond the nuclear norm.

{\rm (iii)} It is worthwhile to discuss the quantity $2\alpha_2-\mu$ appearing in the denominators of the recovery bounds in Theorem \ref{thm-stat}. One may think that the constants $\alpha_2$ and $\mu$ play opposite roles in the estimation method. Actually, $\alpha_1$ measures the curvature of the nonconvex loss function $\loss_N$ and $\mu$ measures the nonconvexity of the regularizer $\regu_\lambda$. Larger values of $\mu$ result in more severe nonconvexity of the regularizer and thus serve as an obstacle in estimation, while larger values of $\alpha_2$ can lead to a well-behaved estimator. These two constants interact together to balance the noncovexity of the estimator \eqref{eq-esti} and therefore, the requirement that $2\alpha_2>\mu$ is used to control this confrontational relationship and finally ensure a good performance of local optima. 
\end{Remark}

\section{Consequences for errors-in-variables matrix regression}\label{sec-conse}


In this section, we consider high-dimensional errors-in-variables matrix regression with noisy covariates as an application of the nonconvex regularized method. The matrix regression model has been first proposed in \cite{wainwright2014structured}. Fruitful results on both theoretical and applicational aspects have been developed in the last decade; see, e.g., \cite{alquier2013rank,candes2010power,negahban2011estimation,recht2010guaranteed} and references therein. Consider the following generic observation model
\begin{equation}\label{eq-matrix}
y_i=\langle\langle X_i,\Theta^* \rangle\rangle+\epsilon_i,\quad i=1,2,\cdots,N,
\end{equation}
where for two matrices $X_i, \Theta^*\in \R^{d_1\times d_2}$, their inner product is defined as $\langle\langle X_i,\Theta^* \rangle\rangle:=\text{trace}(\Theta^*X_i^\top)$, $\{y_i\}_{i=1}^N$ are observation values for the response variable, $\{X_i\}_{i=1}^N$ are covariate matrices, $\{\epsilon_i\}_{i=1}^N$ are observation noise, $\Theta^*$ is the unknown parameter matrix to be recover. Model \eqref{eq-matrix} can be written in a more compact form using operator-theoretic notations as follows 
\begin{equation}\label{eq-matrix-op}
y=\X(\Theta^*)+\epsilon,
\end{equation}
where $y:=(y_1,y_2,\cdots,y_N)^\top\in \R^N$ is the response vector, $\epsilon:=(\epsilon_1,\epsilon_2,\cdots,\epsilon_N)^\top\in \R^N$, the operator $\X:\R^{d_1\times d_2}\to \R^N$ is given by the covariate matrices $\{X_i\}_{i=1}^N$ as $[\X(\Theta^*)]_i=\langle\langle X_i,\Theta^* \rangle\rangle$.

In standard analysis, the observations $\{(X_i,y_i)\}_{i=1}^N$ are usually assumed to be correctly obtained. However, this assumption is neither reasonable nor realistic in many real applications, since the covariates are always with noise or miss values. Hence one can only observe the corrupted pairs $\{(Z_i,y_i)\}_{i=1}^N$ instead, where $Z_i$'s are noisy observations of the corresponding true $X_i$'s. The following two types of errors are considered:
\begin{enumerate}[(a)]
\item Additive noise: For each $i=1,2,\cdots,N$, we observe $Z_{i}=X_{i}+W_{i}$, with $W_{i}\in \R^{d_1\times d_2}$ being a random matrix.
\item Missing data: For each $i=1,2,\cdots,N$, we observe a random matrix $Z_{i}\in \R^{d_1\times d_2}$, where for each $j=1,2\cdots,d_1$, $k=1,2\cdots,d_2$, $({Z_i})_{jk}=({X_i})_{jk}$ with probability $1-\rho$, and $({Z_i})_{jk}=*$ with probability $\rho$, where $\rho\in [0,1)$.
\end{enumerate}
Throughout this section, a Gaussian-ensemble assumption is imposed on the errors-in-variables matrix regression model. In detail, for a matrix $X\in \R^{d_1\times d_2}$, define $M=d_1\times d_2$ and we  use $\text{vec}(X)\in \R^M$ to denote the vectorized form of the matrix $X$. The $\Sigma$-ensemble is defined as follows.
\begin{Definition}\label{def-sig}
Let $\Sigma\in \R^{M\times M}$ be a symmetric positive definite matrix. Then a matrix $X$ is said to be observed form the $\Sigma$-ensemble if $\emph{vec}(X)\sim \N(0,\Sigma)$.
\end{Definition}
Then for each $i=1,2,\cdots,N$, the true covariate matrix $X_i$ and the noisy matrix $W_i$ are assumed to be drawn from the $\Sigma_x$-ensemble and the $\Sigma_w$-ensemble, respectively. The observation noise $\epsilon$ is assumed to obey the Gaussian distribution with $\epsilon\sim \N(0,\sigma_\epsilon^2\I_N)$.

When there exists no measurement error, one naturally consider the following empirical loss function to recover the true parameter $\Theta^*$
\begin{equation}\label{eq-vec}
\loss_N(\Theta)=\frac{1}{2N}\|y-\X(\Theta)\|_2^2=\frac{1}{2N}\sum_{i=1}^N(y_i-\text{vec}(X_i)^\top\text{vec}(\Theta))^2.
\end{equation}
Define $\XX=(\text{vec}(X_1),\text{vec}(X_2),\cdots,\text{vec}(X_N))^\top\in \R^{N\times M}$. Then \eqref{eq-vec} can be written as
\begin{equation}\label{eq-vec}
\loss_N(\Theta)=\frac{1}{2N}\|y-\XX \text{vec}(\Theta)\|_2^2=\frac{1}{2N}\|y\|_2^2-\frac{1}{N}\langle \XX^\top y, \text{vec}(\Theta)\rangle+\frac{1}{2N}\langle \XX^\top \XX \text{vec}(\Theta), \text{vec}(\Theta)\rangle. 
\end{equation}
However, in the errors-in-variables case, the quantities $\frac{\XX^\top \XX}{N}$
and $\frac{\XX^\top y}{N}$ in \eqref{eq-vec} are both unknown, and thus the loss function does not work. Nonetheless, this loss function still provides us some heuristic. Specifically, given a collection of observations, the plug-in technique in \cite{loh2012high} is to find suitable estimates of the quantities $\frac{\XX^\top \XX}{N}$ and $\frac{\XX^\top Y}{N}$ adaptive to the cases of additive noise and/or missing data. 

Let $(\hat{\Gamma},\hat{\Upsilon})$ be estimators of $(\frac{\XX^\top \XX}{N},\frac{\XX^\top y}{N})$. Recall the nonconvex estimation method proposed in \ref{eq-esti}, we obtain the following estimator in the errors-in-variables case
\begin{equation}\label{eq-esti-error}
\hat{\Theta} \in \argmin_{\Theta\in \Omega\subseteq \R^{d_1\times d_2}}\left\{\frac{1}{2}\langle \hat{\Gamma}\text{vec}(\Theta), \text{vec}(\Theta)\rangle-\langle \hat{\Upsilon}, \text{vec}(\Theta)\rangle+\regu_\lambda(\Theta)\right\},
\end{equation}
where the surrogate loss function is 
\begin{equation}\label{eq-surloss}
\loss_N(\Theta)=\frac{1}{2}\langle \hat{\Gamma}\text{vec}(\Theta), \text{vec}(\Theta)\rangle-\langle \hat{\Upsilon}, \text{vec}(\Theta)\rangle.
\end{equation}

Define $\Z=(\text{vec}(Z_1),\text{vec}(Z_2),\cdots,\text{vec}(Z_N))^\top\in \R^{N\times M}$. For specific additive noise and missing data cases, an unbiased choice of the pair $(\hat{\Gamma},\hat{\Upsilon})$ is given respectively by
\begin{align}
\hat{\Gamma}_{\text{add}}&:=\frac{\Z^\top \Z}{N}-\Sigma_w \quad \mbox{and}\quad \hat{\Upsilon}_{\text{add}}:=\frac{\Z^\top y}{N}, \label{sur-add}\\
\hat{\Gamma}_{\text{mis}}&:=\frac{\tilde{\Z}^\top \tilde{\Z}}{N}-\rho\cdot\text{diag}\left(\frac{\tilde{\Z}^\top \tilde{\Z}}{N}\right) \quad \mbox{and}\quad \hat{\Upsilon}_{\text{mis}}:=\frac{\tilde{\Z}^\top y}{N}\quad \left(\tilde{\Z}=\frac{\Z}{1-\rho}\right). \label{sur-mis}
\end{align}

Under the high-dimensional scenario $(N\ll M)$, it is easy to check that the estimated matrices $\hat{\Gamma}_{\text{add}}$ and $\hat{\Gamma}_{\text{mis}}$ in \eqref{sur-add} and \eqref{sur-mis} are always negative definite; actually, $\hat{\Gamma}_{\text{add}}$ and $\hat{\Gamma}_{\text{mis}}$ are obtained via subtracting the positive definite matrices $\Sigma_{w}$ and $\rho\cdot \text{diag}\left(\frac{\tilde{\Z}^\top \tilde{\Z}}{N}\right)$ from the matrices $\Z^{\top}\Z$ and $\tilde{\Z}^{\top} \tilde{\Z}$ with rank at most $N$, respectively. Therefore, the loss function in \eqref{eq-surloss} is nonconvex and thus the estimator \eqref{eq-esti-error} is based on an optimization problem consisting of a loss function and a regularizer that are both nonconvex. Then the problem of recovering the parameter matrix in errors-in-variables matrix regression falls into the framework of this article.

In order to apply the theoretical result in the previous section, we need to verify that the LSC and RSC conditions \eqref{eq-rsc1} and \eqref{eq-rsc2} hold by the errors-in-variables matrix regression model with high probability. In addition, noting from \eqref{eq-lambda-sta} and \eqref{eq-l2bound}, we shall also bound the term $\normm{\nabla \loss_N(\Theta^*)}_\text{op}$ and $\normm{\Pi_{\Aa_{J_1}}(\nabla \loss_N(\Theta^*))}_\text{op}$ to provide some insight in choosing the regularization parameter and to show the smaller scale of the latter quantity. Specifically, Proposition \ref{prop-add} is for the additive noise case, while Proposition \ref{prop-mis} is for the missing data case.

To simplify the notations, let $\tilde{d}=\max\{d_2,d_2\}$ and define
\begin{align}
\tau_{\text{add}}&:= \lambda_{\text{min}}(\Sigma_x)\max\left\{\frac{(\normm{\Sigma_x}_\text{op}^2+\normm{\Sigma_w}_\text{op}^2}{\lambda^2_{\text{min}}(\Sigma_x)},1\right\},\label{tau-add}\\
\tau_{\text{mis}}&:= \lambda_{\text{min}}(\Sigma_x)\max\left\{\frac{1}{(1-\rho)^4}\frac{\normm{\Sigma_x}_\text{op}^4}{\lambda^2_{\text{min}}(\Sigma_x)},1\right\},\label{tau-mis}\\
\varphi_{\text{add}}&:=(\normm{\Sigma_x}_\text{op}+\normm{\Sigma_x}_\text{op})(\normm{\Sigma_x}_\text{op}+\sigma_\epsilon)\normm{\Theta^*}_\text{F},\label{deviation-add}\\
\varphi_{\text{mis}}&:=\frac{\normm{\Sigma_x}_\text{op}}{1-\rho}\left(\frac{\normm{\Sigma_x}_\text{op}}{1-\rho}+\sigma_\epsilon\right)\normm{\Theta^*}_\text{F}.\label{deviation-mis}
\end{align}
It is worthwhile to note that $\varphi_{\text{add}}$ and $\varphi_{\text{mis}}$ respectively serves as the surrogate error for errors-in-variables matrix regression and plays a key role in bounding the quantities $\normm{\nabla\loss_N(\Theta^*)}_\text{op}$ and $\normm{\Pi_{\Aa_{J_1}}(\nabla \loss_N(\Theta^*))}_\text{op}$.

\begin{Proposition}\label{prop-add}
In the additive noise case, let $\tau_{\emph{add}}$ and $\varphi_{\emph{add}}$ be defined as in \eqref{tau-add} and \eqref{deviation-add}, respectively. Then the following conclusions are true:\\
\emph{(i)} there exist universal positive constants $(c_0,c_1,c_2)$ such that with probability at least $1-c_1\exp\left(-c_2N\min\left\{\frac{\lambda^2_{\emph{min}}(\Sigma_x)}{(\normm{\Sigma_x}_\emph{op}^2+\normm{\Sigma_w}_\emph{op}^2)^2},1\right\}\right)$, the \emph{LSC} condition \eqref{eq-rsc1} holds with parameters
\begin{equation}\label{add-staRSC-para}
\alpha_1=\frac{1}{2}\lambda_{\emph{min}}(\Sigma_x),\ \emph{and}\ \tau_1=c_0\tau_{\emph{add}}\frac{\tilde{d}\log \tilde{d}}{N};
\end{equation}
\emph{(ii)} assume that $N\geq 4c_0\tau_{\emph{add}}\frac{\tilde{d}\log \tilde{d}}{N}$, then there exist universal positive constants $(c_0,c_1,c_2)$ such that with probability at least $1-c_1\exp\left(-c_2N\min\left\{\frac{\lambda^2_{\emph{min}}(\Sigma_x)}{(\normm{\Sigma_x}_\emph{op}^2+\normm{\Sigma_w}_\emph{op}^2)^2},1\right\}\right)$, the \emph{RSC} condition \eqref{eq-rsc2} holds with parameter
\begin{equation}\label{add-algRSC-para}
\alpha_2=\frac{1}{8}\lambda_{\emph{min}}(\Sigma_x);
\end{equation}
\emph{(iii)} there exist universal positive constants $(c_3,c_4,c_5)$ such that
\begin{equation*}
\begin{aligned}
\normm{\nabla\loss_N(\Theta^*)}_\emph{op}&\leq c_3\varphi_{\emph{add}}\sqrt{\frac{\log \tilde{d}}{N}},\\
\normm{\Pi_{\Aa_{J_1}}(\nabla\loss_N(\Theta^*))}_\emph{op}&\leq c_3\varphi_{\emph{add}}\sqrt{\frac{\log r_1}{N}}
\end{aligned}
\end{equation*}
hold with probability at least $1-c_4\exp(-c_5\tilde{d})$ and $1-c_4\exp(-c_5r_1)$, respectively.
\end{Proposition}

\begin{proof}
\rm{(i)} Note that the matrices $X$ and $W$ are sub-Gaussian with parameters $(\Sigma_x,\normm{\Sigma_x}_\text{op}^2)$ and $(\Sigma_w,\normm{\Sigma_w}_\text{op}^2)$, respectively \cite{loh2012supplementaryMH}. Then \cite[Lemma 1]{loh2012high} is applicable to concluding that, there exist universal constants $(c_0,c_1,c_2)$ such that for any $\beta\in \R^M$
\begin{align}
\beta^\top\hat\Gamma_{\text{add}}\beta&\geq \frac{1}{2}\lambda_{\text{min}}(\Sigma_x)\|\beta\|_2^2-c_0\tau_{\text{add}}\frac{\log M}{N}\|\beta\|_1^2,\label{add-RSC1}
\end{align}
with probability at least $1-c_1\exp\left(-c_2N\min\left\{\frac{\lambda^2_{\text{min}}(\Sigma_x)}{(\normm{\Sigma_x}_\text{op}^2+\normm{\Sigma_w}_\text{op}^2)^2},1\right\}\right)$. Let $\Theta^*$ be the true parameter matrix (i.e., $\Theta^*$ satisfies \eqref{eq-matrix-op}). Fix $\Delta\in \R^{d_1\times d_2}$ such that $\Theta^*+\Delta \in \Su$. In the additive noise case, it holds that $\inm{\nabla\loss_N(\Theta^*+\Delta)-\nabla\loss_N(\Theta^*)}{\Delta}=\text{vec}(\Delta)^\top\hat\Gamma_{\text{add}}\text{vec}(\Delta)$.
Now with $\text{vec}(\Delta)$ in palce of $\beta$ and noting the facts that $\|\text{vec}(\Delta)\|_2=\normm{\Delta}_\text{F}$ and that $\|\text{vec}(\Delta)\|_1\leq \sqrt{M}\|\text{vec}(\Delta)\|_1=\sqrt{M}\normm{\Delta}_\text{F}$, one has by \eqref{add-RSC1} that the LSC condition \eqref{eq-rsc1} hold with high probability with parameters given by \eqref{add-staRSC-para} due to the facts that $M=d_1*d_2$ and $\tilde{d}=\max\{d_1,d_2\}$.

\rm{(ii)} On the other hand, for any $\Theta,\Theta'\in \R^{d_1\times d_2}$, note that $\loss_N(\Theta+\Delta)-\loss_N(\Theta)-\inm{\nabla\loss_N(\Theta)}{\Delta}=\frac{1}{2}\text{vec}(\Delta)^\top\hat\Gamma_{\text{add}}\text{vec}(\Delta)$. Then with $\text{vec}(\Delta)$ in place of $\beta$ and noting the facts that $\|\text{vec}(\Delta)\|_2=\normm{\Delta}_\text{F}$ and that $\|\text{vec}(\Delta)\|_1\leq \sqrt{M}\|\text{vec}(\Delta)\|_2=\sqrt{M}\normm{\Delta}_\text{F}$,
one has that
\begin{equation*}
\loss_N(\Theta+\Delta)-\loss_N(\Theta)-\inm{\nabla\loss_N(\Theta)}{\Delta}\geq \frac{1}{4}\lambda_{\text{min}}(\Sigma_x)\normm{\Delta}_{\text{F}}^2-\frac{c_0}{2}\tau_{\text{add}}\frac{\sqrt{M}\log M}{N}\normm{\Delta}_{\text{F}}^2.
\end{equation*}
Combining this inequality with the assumption that $N\geq 4c_0\tau_{\text{add}}\frac{\tilde{d}\log \tilde{d}}{N}$, we obtain that the RSC condition \eqref{eq-rsc2} holds with high probability with parameters given by \eqref{add-algRSC-para} due to the facts that $M=d_1*d_2$ and $\tilde{d}=\max\{d_1,d_2\}$. 

\rm{(iii)} It follows from simple calculation that
\begin{equation*}
\normm{\nabla\loss_N(\Theta^*)}_\text{op}=\|\hat{\Gamma}_{\text{add}}\text{vec}(\Theta^*)-\hat{\Upsilon}_{\text{add}}\|_\infty.
\end{equation*}
Thus \cite[Lemma 2]{loh2012high} is applicable to concluding that, there exist universal positive constants $(c_3,c_4,c_5)$ such that
\begin{equation*}
\normm{\nabla\loss_n(\Theta^*)}_\text{op}\leq c_3\varphi_{\text{add}}\sqrt{\frac{\log \tilde{d}}{N}},
\end{equation*}
holds with probability at least $1-c_4\exp(-c_5\tilde{d})$.

Then by the definition of $\Pi_{\Aa_{J_1}}$ (cf. \eqref{eq-proj1}), one has that 
\begin{equation*}
\normm{\Pi_{\Aa_{J_1}}(\nabla\loss_N(\Theta^*))}_\text{op}=\normm{U_{J_1}^*{U_{J_1}^*}^\top\nabla\loss_N(\Theta^*)V_{J_1}^*{V_{J_1}^*}^\top}_\text{op}=\normm{U_{J_1}^*\nabla\loss_N(\Theta^*)V_{J_1}^*}_\text{op}.
\end{equation*}
Since $U_{J_1}^*\nabla\loss_N(\Theta^*)V_{J_1}^*\in \R^{{r_1}\times {r_1}}$, it follows from \cite[Lemma 2]{loh2012high} that, there exist universal positive constants $(c_3,c_4,c_5)$ such that
\begin{equation*}
\normm{\Pi_{\Aa_{J_1}}(\nabla\loss_N(\Theta^*))}_\text{op}\leq c_3\varphi_{\text{add}}\sqrt{\frac{\log r_1}{N}},
\end{equation*}
holds with probability at least $1-c_4\exp(-c_5r_1)$.
The proof is complete.
\end{proof}

The proof of Proposition \ref{prop-mis} is similar to that of Proposition \ref{prop-add} with the modification that \cite[Lemma 3 and Lemma 4]{loh2012high} is applied instead of \cite[Lemma 1 and Lemma 2]{loh2012high}, and so is omitted.

\begin{Proposition}\label{prop-mis}
In the missing data case,  let $\tau_{\emph{mis}}$ and $\varphi_{\text{mis}}$ be defined as in \eqref{tau-mis} and \eqref{deviation-mis}, respectively. Then the following conclusions are true:\\
\emph{(i)} there exist universal positive constants $(c_0,c_1,c_2)$ such that with probability at least $1-c_1\exp\left(-c_2N\min\left\{(1-\rho)^4\frac{\lambda^2_{\emph{min}}(\Sigma_x)}{\normm{\Sigma_x}_\emph{op}^4},1\right\}\right)$, the \emph{LSC} condition \eqref{eq-rsc1} holds with parameters
\begin{equation*}
\alpha_1=\frac{1}{2}\lambda_{\emph{min}}(\Sigma_x),\ \emph{and}\ \tau_1=c_0\tau_{\emph{mis}}\frac{\tilde{d}\log \tilde{d}}{N};
\end{equation*}
\emph{(ii)} assume that $N\geq 4c_0\tau_{\emph{mis}}\frac{\tilde{d}\log \tilde{d}}{N}$, then there exist universal positive constants $(c_0,c_1,c_2)$ such that with probability at least $1-c_1\exp\left(-c_2N\min\left\{(1-\rho)^4\frac{\lambda^2_{\emph{min}}(\Sigma_x)}{\normm{\Sigma_x}_\emph{op}^4},1\right\}\right)$, the \emph{RSC} condition \eqref{eq-rsc2} holds with parameter
\begin{equation*}
\alpha_2=\frac{1}{8}\lambda_{\emph{min}}(\Sigma_x);
\end{equation*}
\emph{(ii)} there exist universal positive constants $(c_3,c_4,c_5)$ such that
\begin{equation*}
\begin{aligned}
\normm{\nabla\loss_N(\Theta^*)}_\emph{op}&\leq c_3\varphi_{\emph{mis}}\sqrt{\frac{\log \tilde{d}}{N}},\\
\normm{\Pi_{\Aa_{J_1}}(\nabla\loss_N(\Theta^*))}_\emph{op}&\leq c_3\varphi_{\emph{mis}}\sqrt{\frac{\log r_1}{N}}.
\end{aligned}
\end{equation*}
holds with probability at least $1-c_4\exp(-c_5\tilde{d})$ and $1-c_4\exp(-c_5r_1)$, respectively.
\end{Proposition}
\begin{Remark}\label{rmk-prop}
It is obvious to see that the scale of $\normm{\Pi_{\Aa_{J_1}}(\nabla\loss_N(\Theta^*))}_\emph{op}$ is smaller than that of $\normm{(\nabla\loss_N(\Theta^*))}_\emph{op}$ in both additive noise and missing data cases. Hence the recovery bound of nonconvex regularized method can be tighter than that of the nuclear norm regularized method.
\end{Remark}

With these two propositions, we are now ready to provide the probabilistic consequences for errors-in-variables matrix regression. Corollary \ref{corol-add} is for the additive noise case, while Corollary \ref{corol-mis} is for the missing data case. The proofs mainly follow by applying Propositions \ref{prop-add} and \ref{prop-mis} on Theorem \ref{thm-stat} with some elementary probability theory, respectively, and so are omitted.

\begin{Corollary}\label{corol-add}
Let $r, \omega$ be positive numbers such that $\Theta^*\in \Omega$ and satisfies \eqref{eq-rank}. Let $\tilde{\Theta}$ be a stationary point of the optimization problem \eqref{eq-esti} with $\loss_N$ given by \eqref{eq-surloss} and $(\hat{\Gamma}_{\emph{add}},\hat{\Upsilon}_{\emph{add}})$ in place of $(\hat{\Gamma},\hat{\Upsilon})$. Suppose that the nonconvex regularizer $\regu_\lambda$ satisfies Assumption \ref{asup-regu}, and that $\frac{1}{4}\lambda_{\emph{min}}(\Sigma_x)>\mu$. Then there exist universal positive constants $c_i\ (i=0,1,2,3,4,5)$ such that, if $(\lambda,\omega)$ are chosen to satisfy
\begin{equation}\label{eq-lambda-sta}
\lambda\geq c_0\max\left\{\varphi_{\emph{add}}\sqrt{\frac{\log \tilde{d}}{N}},\omega\tau_{\emph{add}}\frac{\tilde{d}\log \tilde{d}}{N}\right\},
\end{equation}
and the sample size satisfies $N\geq 4c_0\tau_{\emph{add}}\frac{\tilde{d}\log \tilde{d}}{N}$,
then then it holds with probability at least $1-c_1\exp\left(-c_2N\min\left\{\frac{\lambda^2_{\emph{min}}(\Sigma_x)}{(\normm{\Sigma_x}_\emph{op}^2+\normm{\Sigma_w}_\emph{op}^2)^2},1\right\}\right)-c_3\exp(-c_4\tilde{d})$ that 
\begin{equation}\label{eq-l2bound}
\normm{\hat{\Theta}-\Theta^*}_\emph{F}\leq \frac{c_5\varphi_{\emph{add}}}{2\alpha_2-\mu}\left(\sqrt{\frac{r_1\log r_1}{N}}+\sqrt{\frac{r_2\tilde{d}\log \tilde{d}}{N}}\right).
\end{equation}
\end{Corollary}

\begin{Corollary}\label{corol-mis}
Let $r, \omega$ be positive numbers such that $\Theta^*\in \Omega$ and satisfies \eqref{eq-rank}. Let $\tilde{\Theta}$ be a stationary point of the optimization problem \eqref{eq-esti} with $\loss_N$ given by \eqref{eq-surloss} and $(\hat{\Gamma}_{\emph{mis}},\hat{\Upsilon}_{\emph{mis}})$ in place of $(\hat{\Gamma},\hat{\Upsilon})$. Suppose that the nonconvex regularizer $\regu_\lambda$ satisfies Assumption \ref{asup-regu}, and that $\frac{1}{4}\lambda_{\emph{min}}(\Sigma_x)>\mu$. Then there exist universal positive constants $c_i\ (i=0,1,2,3,4,5)$ such that, if $(\lambda,\omega)$ are chosen to satisfy
\begin{equation}\label{eq-lambda-sta}
\lambda\geq c_0\max\left\{\varphi_{\emph{mis}}\sqrt{\frac{\log \tilde{d}}{N}},\omega\tau_{\emph{mis}}\frac{\tilde{d}\log \tilde{d}}{N}\right\},
\end{equation}
and the sample size satisfies $N\geq 4c_0\tau_{\emph{mis}}\frac{\tilde{d}\log \tilde{d}}{N}$,
then then it holds with probability at least $1-c_1\exp\left(-c_2N\min\left\{(1-\rho)^4\frac{\lambda^2_{\emph{min}}(\Sigma_x)}{\normm{\Sigma_x}_\emph{op}^4},1\right\}\right)-c_3\exp(-c_4\tilde{d})$ that 
\begin{equation}\label{eq-l2bound}
\normm{\hat{\Theta}-\Theta^*}_\emph{F}\leq \frac{c_5\varphi_{\emph{mis}}}{2\alpha_2-\mu}\left(\sqrt{\frac{r_1\log r_1}{N}}+\sqrt{\frac{r_2\tilde{d}\log \tilde{d}}{N}}\right).
\end{equation}
\end{Corollary}

\begin{Remark}\label{rmk-corol}
\cite{negahban2011estimation} and \cite{Li2024LowrankME} have respectively studied the problem of low-rank matrix recovery in both clean and noisy covariate case based on the nuclear norm regularized method. The recovery bound there scales as 
\begin{equation}\label{eq-wain}
\normm{\hat{\Theta}-\Theta^*}_\emph{F}=O\left(\sqrt{\frac{r\tilde{d}}{N}}\right).
\end{equation} 
When $r_1>0$, meaning that there are $r_1$ singular values larger than $\nu$, the recovery bounds in Corollaries \ref{corol-add} and \ref{corol-mis} are both tighter than \eqref{eq-wain} as long as the sample size $N=\Omega(\tilde{d}\log \tilde{d})$. When $r_1=r$, meaning that all the singular values are larger than $\nu$, the nonconvex estimator achieves the best-case recovery bound scaling as $\normm{\hat{\Theta}-\Theta^*}_\emph{F}=O\left(\sqrt{\frac{r\log r}{N}}\right)$. 

\end{Remark}

\section{Conclusion}\label{sec-con}

In this work, we considered a nonconvex regularized method for high-dimensional low-rank matrix recovery. Under suitable regularity conditions on the nonconvex loss function and regularizer, we provided the recovery bound for any stationary point of the nonconvex method. Furthermore, when some of the singular values of the parameter matrix are larger than a threshold given by the nonconvex regularizer, the established recovery bound is much tighter than that of the convex nuclear norm regularized method. In addition, the theoretical result was applied on errors-in-variables matrix regression to obtain probabilistic consequences via verifying the required conditions.


\bibliographystyle{plain}  
\bibliography{./exactEIV-ref}

\end{document}